\documentclass[leqno,sumlimits,intlimits,namelimits,draft]{amsart}

\usepackage{amssymb,latexsym}
\usepackage[numeric,initials,bibtex-style,nobysame]{amsrefs}
\usepackage{amsthm}

\usepackage{a4wide}
\usepackage[USenglish]{babel}

\usepackage{eucal}
\usepackage{mathrsfs}

\newcommand{\mb}[1]{\ensuremath{\mathbb{#1}}}
\newcommand{\N}{\mb{N}}

\newcommand{\R}{\mb{R}}

\newcommand{\Z}{\mb{Z}}
\newcommand{\T}{\mb{T}}

\renewcommand{\d}{\ensuremath{\partial}}
\newcommand{\diff}[1]{\frac{d}{d#1}}

\newfont{\bl}{msbm10 scaled \magstep2}

\newtheorem{theorem}{Theorem}[section]
\newtheorem{lemma}[theorem]{Lemma}
\newtheorem{proposition}[theorem]{Proposition}

\newtheorem{corollary}[theorem]{Corollary}

\theoremstyle{definition}
\newtheorem{remark}[theorem]{Remark}

\newcommand{\beq}{\begin{equation}}
\newcommand{\eeq}{\end{equation}}
\newcommand{\col}{\colon}
\newcommand{\FT}[1]{\widehat{#1}}
\newcommand{\inp}[2]{\langle #1 | #2 \rangle}  %math mode
\newcommand{\notmid}{\mid\kern-0.5em\not\kern0.5em}

\newcommand{\norm}[2]{{\| #1 \|}_{#2}}

\newcommand{\ga}{\gamma}

\newcommand{\eps}{\varepsilon}

\begin{document}

\pagestyle{plain}

\title{Wave breaking of periodic solutions to the Fornberg-Whitham equation}

\author{G\"unther H\"ormann}

\address{Fakult\"at f\"ur Mathematik\\
Universit\"at Wien, Austria}

\email{guenther.hoermann@univie.ac.at}

%\thanks{}



\date{\today}

\begin{abstract}
Based on recent well-posedness results in Sobolev (or Besov spaces) for periodic solutions to the Fornberg-Whitham equations we investigate here the questions of wave breaking and blow-up for these solutions. We show first that finite maximal life time of a solution necessarily leads to wave breaking. Second, we prove that for a certain class of initial wave profiles the corresponding solutions do indeed blow-up in finite time. 
\end{abstract}

\maketitle

%%%%%%%%%%%%%%%%%%%%%%%%%%%%%%%%%%
%%%%%%%%%%%%%%%%%%%%%%%%%%%%%%%%%%
%%%%%%%%%%%%%%%%%%%%%%%%%%%%%%%%%%

\section{Introduction}

We investigate the qualitative properties regarding blow-up and wave breaking of (spatially) periodic solutions to the so-called Fornberg-Whitham equation, which was introduced as a shallow water wave model that is comparably simple and yet showed indications of wave breaking (cf.\ \cites{Seliger68,Whitham1974,FB78,NaumShish94}). For non-periodic solutions to the Fornberg-Whitham equation on the real line, rigorous blow-up results and wave breaking have been proved in  \cite{ConstantinEscher1998}. Here, we show that similar results hold also in the periodic case.

Let $\T = \R / \Z$ be the one dimensional torus group. Functions on  $\T$ may be identified with $1$-periodic functions on $\R$. We will consider the wave height described basically by a function of space and time $u \col \T \times \R \to \R$, $(x,t) \mapsto u(x,t)$, though in the relevant cases of a finite time of existence of a wave solution $u$, the time domain will be confined to a bounded closed interval $[0,T_0]$, $T_0 > 0$, or to a half-open interval $[0,T[$, $T > 0$. We will often write $u(t)$ to denote the function $x \mapsto u(x,t)$.

The Cauchy problem for the Fornberg-Whitham equation reads
\begin{align} \label{FW3rdOrder}
     u_{txx} - u_t + \frac{9}{2} u_x u_{xx} + \frac{3}{2} u u_{xxx} - 
     \frac{3}{2} u u_x + u_x  &= 0,\\
     u(x,0) &= u_0(x).
\end{align}
If we suppose that $u$ is at least continuous with respect to the time variable and of some Sobolev regularity with respect to the spatial variable, then we may employ the linear continuous operator $Q := (\text{id} - \d_x^2)^{-1} \col H^r(\T) \to H^{r+2}(\T)$, for any $r \in \R$ (acting on the spatial part $u(.,t)$ at any fixed time $t$), to rewrite the above Cauchy problem in the following non-local form
\begin{align}
   \label{FWEqu}  u_t + \frac{3}{2} u u_x &= Q u_x ,\\
     \label{IC} u(x,0) &= u_0(x),
\end{align}
which also requires less spatial regularity of a prospective solution $u$.  

\begin{remark} Concerning normalization factors and signs we followed here in Equations \eqref{FWEqu} and \eqref{FW3rdOrder} the detailed form as used in the recent publications \cites{Holmes16,HolTho17}, which agrees with Equation (29) in \cite{FB78} (in case $\nu = 1$). Note that there was a sign error in the latter equation with the linear term involving $u_x$ (compare with Equation (4) in  \cite{FB78} or with earlier papers cited there). However, replacing $u(x,t)$ by $-u(x,-t)$ transforms solutions of either sign variant of the equation into solutions for the other. Moreover, if $u$ solves \eqref{FWEqu}, then $v := 3 u/2$ is a solution to $v_t + v v_x = Q v_x$. 
\end{remark}

Well-posedness results for (\ref{FWEqu}-\ref{IC}) with spatial regularity according to Sobolev or Besov scales have been obtained in \cites{Holmes16,HolTho17}. Here we will make use of the following simpler consequence:\\ 
\emph{If $s > 3/2$ and $u_0 \in H^s(\T)$, then there exists $T_0 > 0$ such that (\ref{FWEqu}-\ref{IC}) possesses a 
\beq\label{uniquesol}
  \text{unique solution } u \in C([0,T_0], H^s(\T)) \cap C^1([0,T_0],H^{s-1}(\T)). 
\eeq
The map $u_0 \mapsto u$ is  continuous $H^s(\T) \to C([0,T_0], H^s(\T))$ and 
\beq\label{boundedsobolev}
   \sup_{t \in [0,T_0]} \norm{u(t)}{H^s(\T)} < \infty.
\eeq
}
The life span $T_0$ can be guaranteed to be above some a priori lower bound depending only on $s$ and the $H^s$-norm of $u_0$. For given and fixed Sobolev index $s > 3/2$ and $u_0 \in H^s(\T)$ we will consider the \emph{maximal life span} $T$ for a unique solution $u$, defined as the supremum of all possible $T_0$ in the above well-posedness result. Thus, a unique solution that is global in time corresponds to the case $T = \infty$.

The observations made in \cite{HolTho17}*{Theorem 1.5 and its proof} include the following result, which can be considered preparatory of a blow-up situation: \emph{Let $s > 3/2$ and $T$ be the maximal life span $(0 < T \leq \infty)$ for the solution $u \in C([0,T[, H^s(\T))$ to (\ref{FWEqu}-\ref{IC}) corresponding to the initial data $u_0 \in H^s(\T)$. If $T < \infty$, then
\beq\label{prepblowup}
   \limsup_{t \uparrow T} \norm{u(t)}{H^s(\T)} = \infty  \quad \text{and} \quad
   \int_0^T \norm{u_x(t)}{L^\infty(\T)} \, dt = \infty.
\eeq
}

 Section 2 will be concerned with the detailed proof that a finite maximal life span $T < \infty$ for a solution $u$ necessarily implies \emph{wave breaking} for this solution at time $T$. Recall (cf.\ \cite[Definition 6.1]{Constantin2011}) that wave breaking is said to occur for $u$ at time $T > 0$, if the wave itself remains bounded while its slope becomes unbounded, i.e., 
\beq\label{wavebreaking}
  \sup_{t \in [0,T[} \norm{u(t)}{L^\infty(\T)} < \infty \quad \text{and} \quad 
  \limsup_{t \uparrow T} \norm{u_x(t)}{L^\infty(\T)} = \infty.
\eeq
In Section 3 we show that for a certain class of initial configurations $u_0$, the maximal time of existence is indeed finite, hence wave breaking does occur for these initial values.

In the sequel, we will occasionally simplify notation by dropping $\T$ in referring to the spaces  $H^r(\T)$ or $L^p(\T)$.
 
%%%%%%%%%%%%%%%%%%%%%%%%%%%%%%%%%%%%%%%%%%
%%%%%%%%%%%%%%%  SECTION 2  %%%%%%%%%%%%%%%%%%%%%
%%%%%%%%%%%%%%%%%%%%%%%%%%%%%%%%%%%%%%%%%%

\section{Wave breaking in case of finite maximal life span} 

In this section, we will show that the assumption of a finite maximal life span always implies wave breaking. Until stated otherwise, we will use the following convention throughout:
\begin{center}
\emph{Let $s > 3/2$, $u_0 \in H^s$, $u$ be the corresponding unique solution to (\ref{FWEqu}-\ref{IC}), and denote by $T$ its maximal life span.}
\end{center}

We begin by drawing a simple immediate consequence from \eqref{prepblowup}.
\begin{proposition} If $T < \infty$, then
\beq\label{uxunbounded}
   \limsup_{t \uparrow T} \norm{u_x(t)}{L^\infty} = \infty.
\eeq  
\end{proposition}
\begin{proof} 
 For every $T_0 \in \R$ with $0 < T_0 < T$, we have $u_x \in C([0,T_0],H^{s-1}) \subset C([0,T_0],L^\infty)$, since $s-1 > 1/2$ implies $H^{s-1}(\T) \subset L^\infty(\T)$.  We conclude that $\int_0^{T_0} \norm{u_x(t)}{L^\infty} \, dt < \infty$ and the second part of \eqref{prepblowup} then yields
$$
    \forall T_0 \in \R, 0 < T_0 < T: \quad \int_{T_0}^T \norm{u_x(t)}{L^\infty} \, dt = \infty.
$$
Therefore, the continuous function $t \mapsto \norm{u_x(t)}{L^\infty}$, $[0,T[\, \to [0,\infty[$ is unbounded on every interval $[T_0,T[$ with $0 < T_0 < T$. In other words, for every $n \in \N$, $n > 1/T$, we can find $t_n \in [T - \frac{1}{n}, T[$ such that $\norm{u_x(t_n)}{L^\infty} > n$, which proves \eqref{uxunbounded}.
\end{proof}

To obtain a wave breaking result, we need to show that $\norm{u(t)}{\infty}$ remains bounded as $t$ approaches $T$ from the left. As a preparation we first prove boundedness for the $L^2$ norms on finite time intervals.

\begin{lemma} For every $t \in [0,T[$ we have
\beq\label{L2bound}
   \norm{u(t)}{L^2} \leq e^{t} \norm{u_0}{L^2}.
\eeq
\end{lemma}
\begin{proof} We know that 
$$
   u_t(t) + \frac{3}{2} u(t) u_x(t) = Q (u_x(t))
$$
holds with equality in $L^2(\T)$ for every $t$. We may multiply the above equation by $u(t)$  (note that $u(t)^2 u_x(t) \in H^1 \cdot H^1 \cdot L^2 \subset H^1 \cdot L^2 \subset L^2 \cdot L^2 \subset L^1$, since $H^1(\T)$ is an algebra) and obtain, with equality in $L^1$,
$$
  \frac{1}{2} \diff{t} \!\left(u(t)^2\right) + 
     \frac{3}{2} \frac{1}{3} \d_x \!\left( u(t)^3\right) = u(t) Q(\d_x u(t)),
$$
which upon spatial integration over $\T$ and noting that $(Q \circ \d_x) u(t) \in H^{s+1} \subset H^1$ gives
$$
    \frac{1}{2} \diff{t} \norm{u(t)}{L^2}^2 = \inp{u(t)}{(Q \circ \d_x) u(t)} \leq
       \norm{u(t)}{L^2} \norm{(Q \circ \d_x) u(t)}{L^2}
       \leq \norm{u(t)}{L^2} \norm{(Q \circ \d_x) u(t)}{H^1}.
$$
The operator $Q \circ \d_x$ is bounded $L^2(\T) \to H^1(\T)$, in fact, $\norm{(Q \circ \d_x) v}{H^1} \leq \norm{v}{L^2}$ holds for any $v \in L^2(\T)$ (as can be seen from the definition of $H^1(\T)$ via Fourier series representation and from Parseval's formula), hence we deduce further that 
$$
   \diff{t} \norm{u(t)}{L^2}^2 \leq \norm{u(t)}{L^2} \norm{u(t)}{L^2} = 
      \norm{u(t)}{L^2}^2,
$$
which implies $\norm{u(t)}{L^2}^2 \leq e^{2  t} \norm{u_0}{L^2}^2$
as claimed.
\end{proof}

Now we are in a position to show boundedness of the $L^\infty$ norm of the solution in case of a finite maximal life span. An interesting aspect of the following proof is that it does make use of the method of characteristics.

\begin{proposition}\label{linfboundedprop} If $T < \infty$, then
\beq\label{Linftybounded}
    \sup_{t \in [0,T[} \norm{u(t)}{L^\infty} < \infty.
\eeq
\end{proposition}
\begin{proof} Let $y \in \T$ and $\tau \in [0,T[$ be arbitrary. We consider the characteristic ordinary differential equation with initial condition for a curve $\ga \col [0,T[\, \to \T$ corresponding to the solution $u$, that is, 
\beq
    \dot \ga(s) = u(\ga(s),s), \quad \ga(\tau) = y.
\eeq
Note  that $u \in C([0,T[, H^s(\T)) \subseteq C([0,T[, C^1(\T))$, since $H^s(\T) \subset C^1(\T)$, hence $(x,t) \mapsto u(x,t)$ is continuous on $[0,T[\, \times \T$  and globally Lipschitz with respect to $x$. By compactness of $\T$ we therefore have a unique global solution $\ga \in C^1([0,T[,\T)$ to the characteristic initial value problem above. 

By standard reasoning we obtain
\begin{multline*}
  \diff{s} \big(u(\ga(s),s)\big) = u_x(\ga(s),s)\, \dot\ga(s) + u_t (\ga(s),s) =
  u_x(\ga(s),s)\, u(\ga(s),s) + u_t (\ga(s),s)\\ = (Q u_x)(\ga(s),s)
\end{multline*} 
and therefore,
$$
   u(y,\tau) = u(\ga(\tau),\tau) = u(\ga(0),0) + \int_0^\tau  (Q u_x)(\ga(s),s) \, ds
     = u_0(\ga(0)) + \int_0^\tau  (Q u_x)(\ga(s),s) \, ds,
$$
which implies
\beq\label{charestimate}
   |u(y,\tau)| \leq \norm{u_0}{L^\infty} + \int_0^\tau  \left|(Q u_x)(\ga(s),s)\right| \, ds.
\eeq
Observe that for every $r \in [0,T[$ we have $Q u_x(r) \in Q(H^{s-1}) = H^{s+1} \subset H^1 \subset L^\infty$ and we may again employ boundedness of $Q \circ \d_x$ as operator $L^2 \to H^1$ to obtain 
$$
   \left|(Q u_x)(\ga(s),s)\right| \leq \norm{Q u_x (s)}{L^\infty} \leq 
      \norm{Q u_x (s)}{H^1} = \norm{(Q \circ \d_x) u(s)}{H^1} \leq \norm{u(s)}{L^2}.
$$
Inserting this into \eqref{charestimate} yields
$$
    |u(y,\tau)| \leq \norm{u_0}{L^\infty} +  \int_0^\tau  \norm{u(s)}{L^2}  \, ds
$$
and applying \eqref{L2bound} then gives
$$
   |u(y,\tau)| \leq \norm{u_0}{L^\infty} + 
    \norm{u_0}{L^2} \int_0^\tau  e^{s} \, ds 
   = \norm{u_0}{L^\infty} + \norm{u_0}{L^2} (e^{ \tau} - 1)
   \leq \norm{u_0}{L^\infty} + \norm{u_0}{L^2} (e^{T} - 1).
$$
Since $y$ and $\tau$ were arbitrary and the upper bound obtained is independent of these ingredients the proof is complete.
\end{proof}

In combination of the previous two propositions we directly conclude as follows.

\begin{corollary}\label{CorWaveBreak} If $T < \infty$ then wave breaking occurs for the solution $u$ at time $T$.
\end{corollary}

We will now prove a more precise result in case of a slightly more regular initial value, namely $u_0 \in H^2(\T)$ in place of $u_0 \in H^s(\T)$ with merely $s > 3/2$. 

\begin{theorem}\label{wavebreakingthm2} If  $s=2$ and $T < \infty$ then the wave solution $u$ breaks with negative infinite slope at time $T$. More precisely, we have 
$$
  \sup_{t \in [0,T[} \norm{u(t)}{L^\infty(\T)} < \infty 
$$  
while 
$$
  \liminf_{t \uparrow T} \left( \inf_{x \in \T} u_x(x,t) \right) = -\infty.
$$
\end{theorem} 
\begin{proof} The boundedness of $\norm{u(t)}{L^\infty(\T)}$ follows from \eqref{Linftybounded}. The strategy of proof is to show that the negation of the last assertion, i.e., 
\beq\label{negassertion}
  \exists M \geq 0 \col \quad u_x(x,t) \geq - M \quad \forall x \in \T\, \forall t \in [0,T[,
\eeq
 leads to the boundedness of $\norm{u(t)}{H^2}$ as $t$ approaches $T$, which then causes a contradiction due to the first part of \eqref{prepblowup} and the fact that $T$ is supposed to be the maximal life span. 

To reach the contradictory conclusion about the $H^2$ norm, we employ the following line of arguments: 
Since $T < \infty$ we deduce from \eqref{L2bound} that $\sup_{t \in [0,T[} \norm{u(t)}{L^2} < \infty$; moreover, an application of the Poincar\'{e}-Wirtinger inequality (\cite[p.\ 312]{Brezis11}) to the $1$-peridodic function $u_x(t) \in H^1(\T)$ (and noting that $\int_{\T} u_x(x,t)\,dx = 0$) gives $\norm{u_x(t)}{L^2} \leq C \norm{u_{xx}(t)}{L^2}$ for some constant $C > 0$. Therefore, in the current situation we may note the validity of the implication 
$$
    \sup_{t \in [0,T[} \norm{u_{xx}(t)}{L^2} < \infty \quad\Longrightarrow\quad
    \sup_{t \in [0,T[} \norm{u(t)}{H^2} < \infty.
$$ 
Thus, it suffices to show that \eqref{negassertion} implies $\sup_{t \in [0,T[} \norm{u_{xx}(t)}{L^2} < \infty$ and the proof will be complete. More precisely, we will show the following\\[2mm]
\emph{Claim:} \enspace \eqref{negassertion} $\quad\Longrightarrow\quad$ $\displaystyle{\sup_{t \in [0,T[} \norm{u_{xx}(t)}{L^2} \leq  e^{(1 + \frac{15 M}{4}) T} \norm{u_0''}{L^2}}$.

\medskip

As a final technical reduction, we note that, due to well-posedness in the solution space $C([0,T[,H^2)$, the inequality asserted above is stable under $H^2$-limits $u_{0,\eps} \to u_0$ (as $\eps \to 0$) of regularizations $u_{0,\eps} \in H^3(\T)$ of the initial value. Therefore, it suffices to establish the claim for $u_0 \in H^3(\T)$,  in which case we have for the solution $u \in C([0,T[,H^3) \cap C^1([0,T[,H^2)$.

We may apply $\d_x^2$ to Equation \eqref{FWEqu} and obtain
$$
  Q u_{xxx} (t) = \d_x^2 \left(u_t(t) + \frac{3}{2} u(t) u_x(t)\right) =
    u_{txx}(t) + \frac{9}{2} u_x(t) u_{xx}(t) + \frac{3}{2} u(t) u_{xxx}(t), 
$$
which holds as an equation in $L^2(\T)$ for every $t \in [0,T[$, since $Q$ maps $L^2$ into $H^2 \subset L^2$, $u_x(t) u_{xx}(t) \in H^2 \cdot H^1 \subset H^1 \subset L^2$, and $u(t) u_{xxx}(t) \in H^3 \cdot L^2 \subset L^\infty \cdot L^2 \subset L^2$. Multiplication of the above equation by $u_{xx}(t)$ and integrating over $\T$ gives
\begin{multline*}
  \inp{u_{xx}(t)}{(Q \circ \d_x)(u_{xx}(t))} =\\ 
  = \frac{1}{2} \diff{t} \norm{u_{xx}(t)}{L^2}^2
    + \frac{9}{2} \int_{\T} u_x(x,t) u_{xx}(x,t)^2 \, dx + 
    \frac{3}{2} \int_{\T} u(x,t) u_{xx}(x,t) u_{xxx}(x,t)\, dx,
\end{multline*}
where the last integrand is of the form $u\, u_{xx} u_{xxx} =  u\, \d_x (u_{xx}^2)/2$ and an integration by parts yields
$$
  \inp{u_{xx}(t)}{(Q \circ \d_x)(u_{xx}(t))} 
  = \frac{1}{2} \diff{t} \norm{u_{xx}(t)}{L^2}^2
    + \frac{15}{4} \int_{\T} u_x(x,t) u_{xx}(x,t)^2 \, dx.
$$
We rewrite this in the form
$$
   \diff{t} \norm{u_{xx}(t)}{L^2}^2 = 2 \inp{u_{xx}(t)}{(Q \circ \d_x)(u_{xx}(t))}
    + \frac{15}{2} \int_{\T} \underbrace{\big(- u_x(x,t) \big)}_{\leq M} u_{xx}(x,t)^2 \, dx
$$
and by \eqref{negassertion} deduce (again using that $Q \circ \d_x$ is bounded $L^2 \to H^1$)
\begin{multline*}
   \diff{t} \norm{u_{xx}(t)}{L^2}^2 \leq 
     2 \norm{u_{xx}(t)}{L^2} \norm{(Q \circ \d_x) (u_{xx}(t))}{L^2} 
     + \frac{15 M}{2} \norm{u_{xx}(t)}{L^2}^2 \leq \\
     \leq 2  \norm{u_{xx}(t)}{L^2}^2 + \frac{15 M}{2} \norm{u_{xx}(t)}{L^2}^2
     = \left(2  + \frac{15 M}{2}\right) \norm{u_{xx}(t)}{L^2}^2.
\end{multline*}
Integration with respect to time and Gronwall's lemma now imply
$$
\forall t \in [0,T] \col \quad 
   \norm{u_{xx}(t)}{L^2}^2 \leq 
   \norm{u_{xx}(0)}{L^2}^2 \, e^{(2  + \frac{15 M}{2}) t} \leq
   \norm{u_0''}{L^2}^2 \, e^{2 (1 + \frac{15 M}{4}) T},
$$
which proves the claim.
\end{proof}

\begin{remark} In retrospect, we have shown in Corollary \ref{CorWaveBreak} that any solution according to \eqref{uniquesol} with a finite maximal life span and with initial value $u_0 \in H^s(\T)$, $s > 3/2$, suffers wave breaking. Under the condition $s \geq 2 > 3/2$ we add in Theorem \ref{wavebreakingthm2} the qualitive information that this breaking wave develops a singularity with \emph{negative infinite slope}. The reason for requiring the slightly higher regularity on the initial data is techniqual, because the proof aims at establishing the crucial inequality $\sup_{t \in [0,T[} \norm{u_{xx}(t)}{L^2} \leq  C \norm{u_0''}{L^2}$ upon operating with a second order spatial derivative on the basic equation \eqref{FWEqu} and using classical energy estimates with some care. The result might still hold with the relaxed condition $s > 3/2$, but we expect that a proof would require a more heavy machinery from function space theory and would be less direct. If we focus on qualitative aspects of the solutions and in view of the equally important blow-up scenario established in the following section, which uses $s \geq 3$, the setting of Theorem \ref{wavebreakingthm2} suffices.
\end{remark}

%%%%%%%%%%%%%%%%%%%%%%%%%%%%%%%%%%%%%%%%%%
%%%%%%%%%%%%%%%  SECTION 3  %%%%%%%%%%%%%%%%%%%%%
%%%%%%%%%%%%%%%%%%%%%%%%%%%%%%%%%%%%%%%%%%

\section{Blow-up for a class of initial data}

We will show below that for a considerable class of initial wave profiles the maximal life span of the solution is finite and blow-up occurs in the form of wave breaking. The result as well as a large part of the reasoning leading to it is similar to corresponding statements and proofs in \cite{ConstantinEscher1998} for the case of the real line and initial data in $H^\infty(\R)$ (matching with the well-posedness result from \cite{NaumShish94}). The main differences now lie in alternative estimates required for the convolution kernel of the operator $Q$ on the torus and in the fact that here we may work with initial data of lower regularity due to the more recent and improved well-posedness results available (\cites{Holmes16,HolTho17}).

As a preparation we collect a few details about the convolution kernel $K$ implementing the translation invariant operator $Q = (\text{id} - \d_x^2)^{-1}$ on the one-dimensional torus $\T$. The operator $Q$ corresponds to the Fourier multiplier $\FT{K}(k) = 1 / (1 + 4 \pi^2 k^2)$ ($k \in \Z$), which satisfies $\FT{K} \in l^1(\Z) \subset l^2(\Z)$. Therefore $K$ possesses the Fourier series representation
$$
   \sum_{k \in \Z} \frac{e^{2 \pi i k x}}{1 + 4 \pi^2 k^2}
$$  
in the sense of $L^2(\T)$ and the Fourier series itself converges uniformly due to a classic theorem by Weierstra{\ss}, since the sum of the $L^\infty$-norms is convergent. Let the uniform limit be denoted by $F \in C(\T)$. Since $K$ and $F$ both belong to $L^2(\T)$ and 
have the same Fourier series, they agree as classes in $L^2(\T)$, in particular $K(x) = F(x)$ for almost all $x \in \T$, and we may identify $K$ as $L^2$ class with the function $F$. We can obtain a more explicit expression for $F$ (hence $K$) by appealing to the Poisson summation formula (cf.\ \cite[Theorem 3.1.17]{Grafakos04}): Consider $f \in L^1(\R) \cap C(\R)$, given by $f(x) = e^{-|x|}$ and its Fourier transform $\FT{f}$ on the real line 
$$
   \FT{f}(\xi) = \int_{\R} e^{- 2 \pi i x \xi} f(x)\, dx = \frac{2}{1 + 4 \pi \xi^2}; 
$$
thus $\FT{f} \in L^1(\R) \cap C(\R)$ as well and $|f(y)| + |\FT{f}(y)| \leq c /(1 + y^2)$ holds with some constant $c > 0$ for all $y \in \R$; by Poisson's summation formula, we arrive at the following equality of $1$-periodic functions on the real line 
$$
      2 \sum_{k \in \Z} \frac{e^{2 \pi i k x}}{1 + 4 \pi^2 k^2} =
        \sum_{k \in \Z}  \FT{f}(x) e^{2 \pi i k x} = \sum_{k \in \Z}  f(x+k)
        =  \sum_{k \in \Z}  e^{-|x+k|}
        \quad  \text{pointwise for every } x \in \R.
$$ 
Therefore, we have in particular for every $x \in \T$,
$$
  F(x) = \sum_{k \in \Z} \frac{e^{2 \pi i k x}}{1 + 4 \pi^2 k^2} = 
  \frac{1}{2} \sum_{k \in \Z}  e^{-|x+k|},
$$
which allows to evaluate the explicit expression
\begin{multline*}
   F(x) = \frac{1}{2} \left( \sum_{l \in \Z, \atop l \geq 1}  e^{-|x - l|} +  
   \sum_{l \in \Z, \atop l \geq 0}  e^{-|x + l|} \right) =
   \frac{1}{2} \left( \sum_{l=1}^\infty  e^{x - l} +  
   \sum_{l = 0}^\infty  e^{- x - l} \right) =\\ 
   = \frac{1}{2} \left( \frac{e^x}{e-1} +  \frac{e^{1-x}}{e-1} \right)
   = \frac{e^x + e^{1-x}}{2(e-1)}.
\end{multline*}
The formula confirms again that $F$ is continuous (since $F(0) = \lim_{t \uparrow 1} F(t)$), but also reveals that $F$ is $C^1$ on $\T \setminus \{0\}$, piecewise $C^1$ on $\T$ in the sense that the derivative is continuous off $x = 0$ and possesses one-sided limits at $x = 0$, and has a non-differentiable peak at $x = 0$. Thus $K$ is absolutely continuous with piecewise continuous derivative $K'$ (possessing one-sided limits at the only point of discontinuity $x=0$). We have 
\beq\label{Kderivative}
   K'(x) = \frac{e^x - e^{1-x}}{2(e-1)} \quad \forall x \neq 0.
\eeq

Another substantial ingredient in the proof of the blow-up result below is an accurate description of the evolution of spatial extrema of functions $v \in C^1([0,T[,H^2(\T))$, which we may transfer to periodic functions without essential changes from the original version proved for functions on the real line in \cite[Theorem 2.1]{ConstantinEscher1998} (see also \cite[Subsection 6.3.2]{Constantin2011} or Escher's lecture in \cite{ConstEschJohnVill2016}). In fact, due to compactness of $\T$ and the embedding $H^2(\T) \subset C^1(\T)$, the assertion about existence of a location where the extremum is attained is obvious in this case, and the statement on differentiability almost everywhere is proven in exactly the same way.
\begin{lemma}\label{inflemma} If $T > 0$ and $v \in C^1([0,T[,H^2(\T))$, then for every $t \in [0,T[$ there is $\xi(t) \in \T$ such that 
$$
  m(t) := \min_{x \in \T} v_x(x,t) = v_x(\xi(t),t). 
$$
The function $m \col [0,T[ \to \R$ is Lipschitz continuous, in particular, differentiable almost everywhere on $]0,T[$, and satisfies
$$
  m'(t) = v_{tx}(\xi(t),t) \quad \text{for almost every } t \in\, ]0,T[.
$$
\end{lemma}
\noindent The same statement clearly holds for the maximum in place of the minimum.

After all these preparations, we formulate and prove the main result.

\begin{theorem} Let $u_0 \in H^3(\T)$ and $u$ be the unique solution to (\ref{FWEqu}-\ref{IC}) with maximal life span $T$. If 
\beq\label{blowupinitialcondition}
   \min_{x \in \T} u_0'(x) + \max_{x \in \T} u_0'(x) < - \frac{2}{3},
\eeq
then $T < \infty$ and we observe wave breaking for $u$ at time $T$.
\end{theorem}
\begin{proof} Since $u_0 \in H^3(\T)$, we have $u \in C^1([0,T[,H^2) \cap C([0,T[,H^3)$ and therefore Lemma  \ref{inflemma} is applicable with appropriate $\xi_1(t)$ and $\xi_2(t)$ in $\T$  ($t \in [0,T[$) to the functions
\begin{align*}
  m_1(t) := \min_{x \in \T} u_x(x,t) = u_x(\xi_1(t),t),\\
  m_2(t) := \max_{x \in \T} u_x(x,t) = u_x(\xi_2(t),t).
\end{align*}
Note that $u_x(t) \in H^2(\T) \subset C^1(\T)$ and $u_{xx}(\xi_j(t),t) = 0$ holds due to the choice of $\xi_j(t)$ ($j=1,2$) as locations of extrema. By periodicity of the $C^1$ function $u(.,t)$,  we necessarily have $m_1(t) \leq 0 \leq m_2(t)$ for every $t \in [0,T[$. (For example,  $m_1(t) > 0$ is absurd, because the function $u(.,t)$ would then have to be strictly increasing and periodic, which contradicts continuity.)

Differentiating Equation \eqref{FWEqu} with respect to $x$ leads to
$$
  u_{xt}(t) + \frac{3}{2} u_x(t)^2 + \frac{3}{2} u(t) u_{xx}(t) = 
    Q u_{xx}(t) = K \ast (u_{xx}(t))
$$
and inserting $x = \xi_j(t)$ then gives (recall that $u_{xx}(\xi_j(t),t) = 0$)
$$
  m_j'(t) + \frac{3}{2} m_j(t)^2 = \int_{\T} K(y) u_{xx}(\xi_j(t) - y)\, dy 
  \qquad \text{for almost every } t \in\, ]0,T[, j=1,2.
$$
The properties of $K$ allow for an integration by parts, hence we obtain $t$-a.e.\ 
\begin{multline*}
   m_j'(t) + \frac{3}{2} m_j(t)^2 = - \int_{\T} K'(y) u_{x}(\xi_j(t) - y)\, dy =\\
   = - \frac{1}{2(e-1)} \int_0^1 e^y u_{x}(\xi_j(t) - y)\, dy +
   \frac{e}{2(e-1)} \int_0^1 e^{-y} u_{x}(\xi_j(t) - y)\, dy \leq \\
   - \frac{m_1(t)}{2(e-1)} \int_0^1 e^y  dy +
   \frac{e\, m_2(t)}{2(e-1)} \int_0^1 e^{-y}  dy = \frac{1}{2} (m_2(t) - m_1(t)),
\end{multline*}
which in turn yields
\begin{align}
  \label{m1inequality} m_1'(t) \leq - \frac{3}{2} m_1(t)^2 + \frac{1}{2} (m_2(t) - m_1(t)),\\
  \label{m2inequality} m_2'(t) \leq - \frac{3}{2} m_2(t)^2 + \frac{1}{2} (m_2(t) - m_1(t)).
\end{align}
Thus, we are now in a situation perfectly analogous with \cite[Theorem 3.2, p.\ 237]{ConstantinEscher1998}, but for convenience of the reader we repeat the remaining steps of the conclusion. 

The sum of the inequalities in \eqref{m1inequality} and \eqref{m2inequality} gives (almost everywhere on $]0,T[$)
$$
  (m_1 + m_2)' \leq - \frac{3}{2}(m_1^2 + m_2^2) + (m_2 - m_1) =
    (m_2 - m_1)(1 + \frac{3}{2}(m_1 + m_2)) - 3 m_2^2.
$$
The function $m_1 + m_2$ is absolutely continuous on $[0,T[$, $m_2 - m_1 \geq 0$ and the hypothesis of the theorem implies at time $t=0$ the condition $1 + \frac{3}{2}\big(m_1(0) + m_2(0)\big) < 0$.  By the above inequality, the corresponding condition must hold for all time, i.e.,
$$
  \forall t \in [0,T[ \col \quad 1 + \frac{3}{2}\big(m_1(t) + m_2(t)\big) < 0,
$$
which we put to use in \eqref{m1inequality} to deduce (a.e.\ on $[0,T[$)
\begin{multline*}
   m_1' \leq - \frac{3}{2} m_1^2 - \frac{1}{2} m_1 + \frac{1}{2} m_2 <
    - \frac{3}{2} m_1^2 - \frac{1}{2} m_1 + \frac{1}{2} \left(-m_1 - \frac{2}{3}\right) = \\
    = - \frac{3}{2} \left(m_1^2 + \frac{2}{3} m_1  + \frac{2}{9}\right) =
    - \frac{3}{2} \left(\Big(m_1 + \frac{1}{3}\Big)^2  + \frac{1}{9}\right) \leq
    - \frac{3}{2} \left(m_1 + \frac{1}{3}\right)^2.
\end{multline*}
Putting $M(t) := m_1(t) + \frac{1}{3}$ we have $M(0) = m_1(0) + \frac{1}{3} < - \frac{2}{3} - m_2(0) + \frac{1}{3} = - \frac{1}{3} - m_2(0) < 0$ and 
$$
   M'(t) = m_1'(t) \leq -\frac{3}{2} M(t)^2 \quad \text{for almost every } t \in \,]0,T[,
$$
which implies that $M(t) < 0$ throughout. Finally, we obtain  a.e.\ with respect to $t$,
$$
   \diff{t} \left(\frac{1}{M(t)}\right) = - \frac{M'(t)}{M(t)^2} \geq \frac{3}{2}
$$
and therefore upon integration for every $t \in [0,T[$,
$$
   \frac{1}{M(t)} \geq \frac{1}{M(0)} + \frac{3}{2} t.
$$
Observing $M(0) < 0$, we conclude that $M(t) \to - \infty$ as $t \to 2/(3 |M(0)|)$. Thus, $T < \infty$,  $\lim_{t \uparrow T} m_1(t) = - \infty$, and from Proposition \ref{linfboundedprop} we know that the $L^\infty$-norm of $u$ stays bounded as $t$ approaches $T$, which proves wave breaking with negative slope unbounded from below.
\end{proof}

\bigskip

%\paragraph{\emph{Acknowledgements:}} The author thanks Adrian Constantin for suggesting the investigation of these questions and for helpful discussions on details of the proofs.

%%%%%%%%%%%%%%%%%%%%%%%%%%%%%%%%%%%%%%%%%%
%%%%%%%%%%%%%%%  References   %%%%%%%%%%%%%%%%%%%%%
%%%%%%%%%%%%%%%%%%%%%%%%%%%%%%%%%%%%%%%%%%

\bibliography{G,LaG}
\bibliographystyle{abbrv}

\end{document}